\documentclass[a4paper]{article}
\usepackage{amsmath,amsthm,amssymb,mathtools,tikz,url,bbm}
\usepackage[margin=32mm]{geometry}
\title{A remark on the approximation of non-negative polynomials by SONC polynomials}
\author{Gennadiy~Averkov\footnote{Brandenburg University of Technology, Platz der Deutschen Einheit 1, 03046 Cottbus, Germany}}

\newcommand{\R}{\mathbb{R}}

\newcommand{\Z}{\mathbb{Z}}

\newcommand{\setcond}[2]{\left\{#1\,:\,#2\right\}}

\newcommand{\ones}{\mathbbm{1}}

\newtheorem*{thm*}{Theorem}
\newtheorem*{claim*}{Claim}
\newtheorem*{rem*}{Remark}

\begin{document}
\maketitle

\begin{abstract} 
	A SONC polynomial is a sum of finitely many non-negative circuit polynomials, whereas  a non-negative circuit polynomial is a non-negative polynomial whose support is a simplicial circuit. 
	We show that there exist non-negative polynomials that cannot be uniformly approximated by SONC polynomials arbitrarily well. 
\end{abstract} 

\newcommand{\vx}{\mathbf{x}}
\newcommand{\supp}{\operatorname{supp}}
\newcommand{\circuits}{\operatorname{circuits}}

\section{Introduction} 

In the algebra $\R[\vx]:=\R[x_1,\ldots,x_n]$ of real polynomials in  $n \in \Z_{\ge 1}$ variables we consider  the infinite-dimensional convex cones 
\begin{align*}
	P_n & := \setcond{f \in \R[\vx]}{f  \ge 0 \ \text{on} \ \R^n} & & \text{and} & 
 C_n & := \setcond{f \in \R[\vx] }{f \ \text{is SONC}}.
\end{align*}
The definition of SONC polynomials is as follows. The \emph{support} $\supp f$ of a polynomial $f \in \R[\vx]$ is the set of all exponent vectors $\alpha = (\alpha_1,\ldots,\alpha_n) \in \Z_{\ge 0}^n$ with the property that the monomial $\vx^{\alpha} := x_1^{\alpha_1} \cdots x_n^{\alpha_n}$ occurs in $f$ with a non-zero coefficient. A \emph{circuit polynomial} is a polynomial in $\R[\vx]$ with the support satisfying $A \subseteq \supp f \subseteq A \cup \{\beta\}$, for some affinely independent subset $A$ of  $2 \Z_{\ge 0}^n$ and an exponent vector $\beta \in \Z_{\ge 0}^n$  in the relative interior of the convex hull of $A$, such that the coefficients of $f$ for the monomials $\vx^{\alpha}$ with $\alpha \in A$ are   positive. The  special case $A = \{\beta\}$ corresponds to a degenerate circuit polynomial with only one term. A \emph{non-negative circuit polynomial} is a circuit polynomial that is non-negative on $\R^n$.
A \emph{SONC polynomial} is a sum of finitely many non-negative circuit polynomials. SONC polynomials were introduced in \cite{iliman2016amoebas} with the aim of developing tractable solution methods in  polynomial optimization. In this context, the authors of  \cite[Sect.~5]{katthan2021unified} asked whether SONC polynomials can uniformly approximate non-negative polynomials arbitrarily well. We show that the answer to this question is negative. 
For a compact subset $K$ of $\R^n$ with non-empty interior, we use the norm
\[
	\|f\|_K := \max_{x \in K} |f(x)|
\]
on the space $\R[\vx]$ to express uniform approximation on $K$.

\begin{thm*}
	Let $K$ be a compact subset of $\R^n$ with non-empty interior and let  $d \in \Z_{\ge 3}$. Then there exists a polynomial $f \in P_n$ of degree $2 d$ that satisfies
	 \[
	 	\inf_{g \in C_n} \| f - g \|_K > 0. 
	 \]
	 Moreover, one can choose $f$ to be a square of a polynomial of degree $d$ from $\R[\vx]$. 
\end{thm*} 


\section{Proof} 

Pick an interior point $a=(a_1,\ldots,a_n)$ of $K$ that satisfies $a_i \ne 0$ for each $i =1,\ldots,n$. The linear isomorphism $f(x_1,\ldots,x_n) \mapsto h(x_1,\ldots,x_n) := f(a_1 x_1, \ldots, a_n x_n)$ of $\R[\vx]$ keeps $P_n$ and $C_n$ invariant and satisfies $f(a_1,\ldots,a_n) = h(1,\ldots,1)$ as well as $\|f\|_K = \|h\|_H$ for 
\[
	H := \setcond{(x_1,\ldots,x_n) \in \R^n}{(a_1 x_1,\ldots, a_n x_n) \in K}.
\] 
Thus, we can assume without loss of generality that the \emph{all-ones vector} $\ones_n := (1,\ldots,1) \in \R^n$ is in the interior of $K$. Fix $u>1$ close enough to $1$ so that $(u^j,\ones_{n-1}) \in K$ holds for every $j=0,1,2,3$ and consider the linear functional $L : \R[\vx] \mapsto \R$, given as a linear combination of four evaluations as follows: 
\[
	L[f] := f(u^0,\ones_{n-1}) - f(u^1,\ones_{n-1}) + f(u^2,\ones_{n-1}) + f(u^3,\ones_{n-1}). 
\] 

\begin{claim*}
	$L[g] \ge 0 $ holds for every $g \in C_n$. 
\end{claim*} 
\begin{proof}[Proof of Claim]
	Non-negative circuit polynomials generate the convex cone $C_n$. Thus, since $L$ is a linear functional it suffices to check $L[g] \ge 0$  for every non-negative circuit polynomial $g$. If $g$ is degenerate, one has $g = c \vx^\beta$ with $ c \in \R_{>0}$ and $\beta = (\beta_1,\ldots,\beta_n) \in 2 \Z_{\ge 0}^n$, then  $L[g] \ge 0$ is easy to verify: one has $L[g] = L[c \vx^\beta] = c L[\vx^\beta]$, where $c>0$ and $L[\vx^\beta] = L[x_1^{\beta_1}] = 1 - u^{\beta_1} + u^{2 \beta_1} + u^{3 \beta_1} > 0$. 
	
	If $g$ is non-degenerate, it has the form 
	\[
		g = \sum_{\alpha \in A} c_\alpha \vx^{\alpha} + c_\beta \vx^\beta
	\]
	with the coefficients $c_\alpha \in \R_{>0}$ ($\alpha \in A$) and $c_\beta \in \R$, where 
 $A$ is a set of at least two affinely indepent exponent vectors in $2 \Z_{\ge 0}^n$ and $\beta \in \Z_{\ge 0}^n$ is an exponent vector in the relative interior of the convex hull of $A$. Using the notation $\alpha = (\alpha_1,\ldots,\alpha_n)$ and $\beta = (\beta_1,\ldots,\beta_n)$ for the components of $\alpha$ and $\beta$, we  express $L[g]$ as 
 \begin{equation} \label{L:g} 
 		L[g] = \sum_{\alpha \in A} c_\alpha L[x_1^{\alpha_1}] + c_\beta  L[x_1^{\beta}]. 
 \end{equation}
Using $\phi(t):= 1 - e^t + e^{2t} + e^{3t}$, equality \eqref{L:g} becomes 
\begin{equation} \label{L:phi}
	L[g] = \sum_{\alpha \in A} c_\alpha \phi( \alpha_1 \ln u) + c_\beta \phi(\beta_1 \ln u), 
\end{equation}
where $\ln$ is the natural logarithm. Note that since $u>1$, one has $\ln u > 0$. It turns out that $\ln \phi(t)$ is convex on $\R_{\ge 0}$. For checking this, we use the inequality $(\ln \phi(t))'' \ge 0$, which amounts to $\phi''(t) \phi(t) - \phi'(t)^2 \ge 0$. The latter is the inequality 
\begin{equation} \label{e:t}
	(-e^t + 4 e^{2t} + 9 e^{3t} ) ( 1 - e^t + e^{2t} + e^{3t} ) - (-e^t + 2 e^{2t} + 3 e^{3t})^2 \ge 0.  
\end{equation}
Via the substitution $y = e^t$,  inequality \eqref{e:t}, which we want to verify for all  $t \in \R_{\ge 0}$,  gets converted to the polynomial inequality 
\[
	p(y):= (-y + 4 y^2 + 9 y^3) (1-y + y^2 + y^3) - (-y+ 2 y^2 + 3 y^3) \ge 0
\]
for  $y \in \R_{\ge 1}$.  The identity $p(y) = y \bigl( (y-1)^4 + 2 (y-1)^2  + 12 (y-1) + 8 \bigr)$ is a witness for the non-negativity of $p(y)$ on $\R_{\ge 1}$. 

Below, we use the convexity of $\ln \phi(t)$ for bounding $L[g]$ from below. The exponent vector $\beta$ can be expressed uniquely as the convex combination
\[
	\beta = \sum_{\alpha \in A} \lambda_\alpha \alpha,
\]
where $\lambda_\alpha \in \R_{>0}$ for each $\alpha \in A$ and $\sum_{\alpha \in A} \lambda_\alpha = 1$. Expressing the sum $\sum_{\alpha \in A} c_\alpha \phi(\alpha_1 \ln u)$ in \eqref{L:phi} as the weighted arithmetic mean 
\(
	\sum_{\alpha \in A} \lambda_\alpha \frac{c_\alpha \phi(\alpha_1 \ln u)}{\lambda_\alpha}
\)
with the weights $\lambda_\alpha$ ($\alpha \in A$) and  estimating the weighted arithmetic mean by the weighted geometric mean, we obtain
\begin{align*}
	L[g] & \ge \prod_{\alpha \in A} \left( \frac{c_\alpha }{\lambda_\alpha}\right)^{\lambda_\alpha} \phi(\alpha_1 \ln u)^{\lambda_\alpha} + c_\beta \phi( \beta_1 \ln u)
		  = \Theta_g \prod_{\alpha \in A} \phi(\alpha_1 \ln u)^{\lambda_\alpha} + c_\beta \phi(\beta_1 \ln u),
\end{align*}
where $\Theta_g := \prod_{\alpha \in A} \left( \frac{c_\alpha }{\lambda_\alpha}\right)^{\lambda_\alpha} $ is the so-called \emph{circuit number} of the circuit polynomial $g$. As $\ln \phi(t)$ is convex on $\R_{\ge 0}$, we have 
\[
	\prod_{\alpha \in A} \phi(\alpha_1 \ln u)^{\lambda_\alpha} \ge \phi \left( \sum_{\alpha \in A} \lambda_\alpha \alpha_1 \ln u \right)  = \phi(\beta_1 \ln u), 
\]
which yields 
\[
	L[g] \ge (\Theta_g + c_\beta) \, \phi( \beta_1 \ln u) 
\]
Since the value $\phi(\beta_1 \ln u) = 1 - u^{\beta_1} + u^{2 \beta_1} + u^{3 \beta_1}$ is obviously positive and  $\Theta_g + c_\beta$ is known to be non-negative when $g$ is a non-negative circuit polynomial \cite[Thm.~3.8]{iliman2016amoebas}, we see that $L[g] \ge 0$. 
\end{proof}

To conclude the proof of our theorem, we observe that for all  $f \in \R[\vx]$ and $g \in C_n$ we have 
\begin{align*}
	\|g - f\|_K & \ge \max_{j \in \{0,1,2,3\}} \, \bigl|g(u^j,\ones_{n-1}) -f (u^j,\ones_{n-1}) \bigr| & & \text{(since $(u^j,\ones_{n-1}) \in K$)}
		\\ & \ge \frac{1}{4} \, \sum_{j=0}^3 \, \bigl|g(u^j,\ones_{n-1}) -f (u^j,\ones_{n-1})\bigr|
		\\ & \ge \frac{1}{4} \, L[g - f] 
		\\ & = \frac{1}{4} \, \bigl(L[g] - L[f] \bigr)
		\\ &  \ge - \frac{1}{4} \, L[f] & & \text{(by Claim)}. 
\end{align*} 
Consequently, 
\(
	\inf_{g \in C_n} \|g - f\|_K \ge - \frac{1}{4} L[f] 
\)
for every $f \in \R[\vx]$. We fix  
\[
	f  =  (x_1 - 1)^{2} (x_1 - u^2)^2 (x_1-u^3)^{2(d-2)} \in P_n. 
\] As   $f$ depends only on $x_1$, it can be evaluated on $\R$ by substituting a real value for $x_1$. By the choice of $f$, we have  $L[f] = f(1) - f(u) + f(u^2) + f(u^3) = - f(u) < 0$. This gives $\inf_{g \in C_n} \|g - f\|_K \ge \frac{1}{4} f(u) > 0$, as desired. 

\begin{rem*}
	It would be interesting to characterize all pairs $(n,d)$, for which the assertion of our theorem is true. 
\end{rem*} 

\bibliographystyle{plain}
\bibliography{lit}

\end{document}